\documentclass{amsart}
\usepackage{tikz, amsmath, amscd, amssymb, amsfonts, stmaryrd, vmargin, amsthm, soul, bm, graphicx, float, exscale, relsize, etoolbox, mathrsfs, url}

\theoremstyle{plain}        \newtheorem{thm}{Theorem}

\theoremstyle{plain}        \newtheorem{pro}[thm]{Proposition}

\theoremstyle{plain}        \newtheorem{lem}[thm]{Lemma}

\theoremstyle{plain}        \newtheorem{cor}[thm]{Corollary}

\theoremstyle{plain}        \newtheorem{rem}[thm]{Remark}

\theoremstyle{plain}        \newtheorem{eg}{Example}

\theoremstyle{plain}

\begin{document}
\title{Holomorphic isometries between products of complex unit balls}
\thanks{The third author is supported by National Science Foundation grant DMS-1412384 and Simons Foundation grant (\#429722 Yuan Yuan).}
\author{Shan Tai Chan, Ming Xiao \& Yuan Yuan}
\address{Department of Mathematics, Syracuse University, Syracuse, NY 13244-1150, USA.}
\email{schan08@syr.edu}
\address{Department of Mathematics, University of Illinois at Urbana-Champaign, IL, 61801, USA}
\email{mingxiao@illinois.edu}
\address{Department of Mathematics, Syracuse University, Syracuse, NY 13244-1150, USA.}
\email{yyuan05@syr.edu}
\maketitle

\begin{abstract}
We first give an exposition on holomorphic isometries from the Poincar\'e disk to polydisks and from the Poincar\'e disk to the product of the Poincar\'e disk with a complex unit ball. As an application, we provide an example of proper holomorphic map from the unit disk to the complex unit ball that is irrational, algebraic and holomorphic on a neighborhood of the closed unit disk. We also include some new results on holomorphic isometries.
\end{abstract}


\bigskip

Motivated by problems in Number Theory, Clozel-Ullmo \cite{CU03} studied commutators of modular correspondences on a quotient of any irreducible bounded symmetric domain by the torsion-free lattice of automorphisms of the irreducible bounded symmetric domain. In \cite{CU03}, Clozel-Ullmo showed that a germ of holomorphic isometry from the Poincar\'e disk $\Delta$ into polydisks $\Delta^p$, $p\ge 2$, is totally geodesic provided that the image of the isometry is invariant under certain automorphisms of the target polydisk and they
 further conjectured that the assumption of invariance of the image of such holomorphic isometries could be removed.
However, Mok \cite{Mok12} constructed a non-standard (i.e., not totally geodesic) holomorphic isometry, called the $p$-th root embedding, from the Poincar\'e disk into polydisks with respect to their Bergman metrics, which disproved the conjecture.  In general, holomorphic isometries $F: \Delta\rightarrow \Delta^p$ and $G: \Delta \rightarrow \Delta^q$ may induce a holomorphic isometry $H=G^{\#}\circ F: \Delta \rightarrow \Delta^{p+q-1}$, where $G^{\#}: \Delta^p \rightarrow \Delta^{p+q-1}$ is given by $G^{\#}(z_1, \cdots, z_p)=(z_1, \cdots, z_{p-1}, G(z_p))$. It is speculated that such construction from the $p$-th root embeddings and linear diagonal embeddings would exhaust all holomorphic isometries from the Poincar\'e disk to polydisks. The classification problem of holomorphic isometries from the Poincar\'e disk to polydisks has been intensively studied by Mok \cite{Mok12}, Ng \cite{Ng10} and Chan \cite{Ch16a, Ch16b} and we will give an expository of the recent development in section 1.

When the irreducible bounded symmetric domain is the complex unit ball of complex dimension at least 2, Clozel-Ullmo's problem is solved by Mok \cite{Mok02}. More precisely, Mok proved that if $n \geq 2$, any holomorphic isometry from the complex unit ball $\mathbb{B}^n$ to the product of $\mathbb{B}^n$ with respect to the Bergman metric is totally geodesic \cite{Mok02}. Mok's theorem is generalized by Ng \cite{Ng11} and Yuan-Zhang \cite{YZ12} from the complex unit ball of complex dimension at least 2 to the product of complex unit ball of possibly different dimensions. When the domain is of complex dimension 1, namely, from the Poincar\'e disk to the product of complex unit balls, the holomorphic isometry is recently studied by Chan-Yuan \cite{CY17}. A special case from the Poincar\'e disk to the product of the disk with a complex unit ball is understood in \cite{CY17}. However, the understanding of the general situation seems to be far from complete. We will give a survey on this problem in section 2. 

Proper holomorphic maps between complex unit balls is a vital subject in several complex variables and have deep connections to many other mathematical problems. The problem has been studied by many authors and the interested readers may refer to the literature such as \cite{Fo92},  \cite{D93}, \cite{HJ07}, \cite{HJY07}, \cite{NTY16} and references therein. One powerful and successful method is the Chern-Moser normal form in CR geometry. However, this method is not applicable for proper holomorphic map from the unit disk. Various phenomena for proper holomorphic maps from the unit disk to the complex unit ball were discussed by D'Angelo-Huo-Xiao in \cite{DHX16} and we provide a new example in section 3 motivated by their study, indicating as well the significant difference from the higher dimensional case.

We further provide some new results in section 4 on holomorphic isometries including a rigidity result between the products of complex unit balls and a classification result from the Poincar\'e disk to bidisks with conformal constants. Both results are motivated by the paper of Ng \cite{Ng10}. Ng obtained a classification theorem on holomorphic isometries between polydisks in \cite{Ng10} and we derive a higher dimensional analogy between products of complex unit balls. The holomorphic isometries from the Poincar\'e disk to bidisks are classified by Ng \cite{Ng10} and we establish a generalization for arbitrary conformal constants.

This article mostly concerns the particular case of holomorphic isometries between products of complex unit balls.  The interested readers may refer to \cite{Mok12}, \cite{Eb15},\cite{FHX16}, \cite{Yu16} for the general settings of holomorphic maps between Hermitian symmetric spaces.

Throughout this article, denote by $ds^2_U$ the Bergman metric of a bounded domain $U \subset \mathbb{C}^n$; denote by $g_\Omega$ the canonical K\"ahler-Einstein metric on an irreducible bounded symmetric domain $\Omega \subset \mathbb{C}^n$ normalized so that minimal disks are of constant Gaussian curvature $-2$ and by $\omega_\Omega$ the corresponding K\"ahler form.

Let $D\Subset \mathbb C^n$ and $\Omega\Subset \mathbb C^N$ be bounded symmetric domains.
We write $D=D_1\times\cdots \times D_m$ and $\Omega=\Omega_1\times\cdots \times \Omega_k$, where $D_j$, $1\le j\le m$, (resp. $\Omega_l$, $1\le l\le k$), are irreducible factors of $D$ (resp. $\Omega$).
Let $F_1=(F_{1,1}, \cdots, F_{1,k})$, $F_2=(F_{2,1},\ldots,F_{2,k}):(D_1,\lambda_1g_{D_1})\times\cdots\times (D_m,\lambda_m g_{D_m})
\to (\Omega_1,\lambda'_1g_{\Omega_1})\times\cdots \times (\Omega_k,\lambda'_kg_{\Omega_k})$ be holomorphic isometries for some positive real constants $\lambda_j$, $1\le j\le m$, and $\lambda'_l$, $1\le l\le k$,
in the sense that
\[ \bigoplus_{j=1}^m \lambda_j g_{D_j} = \sum_{j=1}^k \lambda'_j F_{l,j}^* g_{\Omega_j}\]
for $l=1,2$.
Then, $F_1$ and $F_2$ are said to be congruent to each other (or $F_1$ equals $F_2$ up to reparametrizations) if there exist $\phi\in \mathrm{Aut}(D)$ and $\Psi\in \mathrm{Aut}(\Omega)$ such that $F_1= \Psi\circ F_2 \circ \phi$; otherwise, $F_1$ and $F_2$ are said to be incongruent to each other.
This defines an equivalence class of holomorphic maps from $(D_1,\lambda_1g_{D_1})\times\cdots\times (D_m,\lambda_m g_{D_m})$ to $(\Omega_1,\lambda'_1g_{\Omega_1})\times\cdots \times (\Omega_k,\lambda'_kg_{\Omega_k})$.

\smallskip

\section{Holomorphic isometries from the Poincar\'e disk into polydisks}
Let $\mathcal H=\{\tau\in \mathbb C\mid \mathrm{Im}\tau>0\}$ be the upper half-plane. 
Mok \cite{Mok12} defined a map $\rho_p:\mathcal H\to \mathcal H^p$ by
\[ \rho_p(\tau)=(\tau^{1\over p},\gamma \tau^{1\over p},\ldots, \gamma^{p-1} \tau^{1\over p}) \]
and showed that $\rho_p:(\mathcal H,ds_{\mathcal H}^2)\to (\mathcal H^p, ds_{\mathcal H^p}^2)$ is a non-standard holomorphic isometry, where $\gamma:=e^{\pi i \over p}$ and $\tau^{1\over p}=r^{1\over p} e^{i\theta\over p}$ for $\tau=r e^{i\theta}\in \mathcal H$.
By composing with Cayley transform from $\Delta$ to $\mathcal H$ (resp. from $\mathcal H^p$ to $\Delta^p$), we obtain a holomorphic isometry $F_p: (\Delta,ds_{\Delta}^2)\to (\Delta^p, ds_{\Delta^p}^2)$, which is called the $p$-th root embedding.
In \cite{Ng10}, we also have the diagonal embedding $\varpi_p:(\Delta,pds_{\Delta}^2)\to (\Delta^p, ds_{\Delta^p}^2)$ given by $\varpi_p(z)=(z,\ldots,z)$ for $p\ge 2$.

Motivated by Mok's fundamental example, Ng \cite{Ng10} studied holomorphic isometries from the Poincar\'e disk into polydisks with respect to their Bergman metrics up to normalizing constants in a systematic way.
Let $f=(f^1,\ldots,f^p):(\Delta,\lambda ds_\Delta^2)\to (\Delta^p,ds_{\Delta^p}^2)$ be a holomorphic isometry such that all component functions of $f$ are non-constant functions, where $\lambda>0$ is real constant and $p\ge 2$ is an integer.
In \cite{Ng10}, Ng showed that $\lambda$ is an integer such that $1\le \lambda \le p$ and introduced the sheeting numbers to $f$ and its component functions.
More precisely, Ng \cite{Ng10} showed that there is an irreducible projective-algebraic curve $V\subset \mathbb P^1 \times (\mathbb P^1)^p$ extending the graph $\mathrm{Graph}(f)\subset \Delta \times \Delta^p$ and there are irreducible projective-algebraic curves $V_j\subset \mathbb P^1 \times \mathbb P^1$ extending the graph $\mathrm{Graph}(f^j)\subset \Delta \times \Delta$ for $j=1,\ldots,p$. Moreover, Ng \cite{Ng10} constructed finite branched coverings $\pi: V\to \mathbb P^1$, $(w,z_1,\ldots,z_p)\mapsto w$, and 
$\pi_j:V_j\to \mathbb P^1$, $(w,z_j)\mapsto w$, in terms of the inhomogeneous coordinates for $j=1,\ldots,p$.
Let $P_j: V \to \mathbb P^1 \times \mathbb P^1$ be the canonical projection given by $P_j(w,z_1,\ldots,z_p) = (w,z_j)$ in terms of the inhomogeneous coordinates for $j=1,\ldots,p$.
In \cite{Ng10}, Ng defined $V_j:=P_j(V)$ and obtained $\pi=\pi_j\circ P_j$ for $j=1,\ldots,p$.

We now follow some notations and settings in \cite{Mok11}.
One may define the normalization $\nu_j:X_j\to V_j$ for $j=1,\ldots,p$ and the normalization $\nu:X\to V$, where $X$ and $X_j$, $j=1,\ldots,p$, are non-singular models of $V$ and $V_j$, $j=1,\ldots,p$, respectively.
Then, we obtain an $n$-sheeted branched covering $\tau:=\pi\circ\nu:X\to \mathbb P^1$ and a $s_j$-sheeted branched covering $\tau_j:=\pi_j\circ \nu_j:X_j\to \mathbb P^1$ for $j=1,\ldots,p$.
We call $n$ the global sheeting number of $f$ and $s_j$ the sheeting number of $f^j$ for $j=1,\ldots,p$. We may introduce the space ${\bf HI}(\Delta,\Delta^p)$ of all holomorphic isometries from $(\Delta, k ds_\Delta^2)$ to $(\Delta^p,ds_{\Delta^p}^2)$ for some positive integer $k$, $1\le k\le p$, and the space ${\bf HI}_k(\Delta,\Delta^p)$ of all holomorphic isometries from $(\Delta, k ds_\Delta^2)$ to $(\Delta^p,ds_{\Delta^p}^2)$. Similarly, we let ${\bf HI}_k(\Delta,\Delta^p;n)$ be the space of all holomorphic isometries from $(\Delta, k ds_\Delta^2)$ to $(\Delta^p,ds_{\Delta^p}^2)$ with the global sheeting number equal to $n$.
Given positive integers $n,s_1,\ldots,s_p$, we denote by ${\bf HI}_k(\Delta,\Delta^p;n;s_1,\ldots,s_p)$ the space of all holomorphic isometries $f=(f^1,\ldots,f^p)$ from $(\Delta, k ds_\Delta^2)$ to $(\Delta^p,ds_{\Delta^p}^2)$ such that the global sheeting number of $f$ is equal to $n$ and the sheeting number of $f^j$ is equal to $s_j$ for $j=1,\ldots,p$. 
In \cite{Ng10}, Ng proved that ${p\over k} \le n\le 2^{p-1}$, $\sum_{j=1}^p {1\over s_j} = k$ and $s_j|n$ for $j=1,\ldots,p$.
Thus, the study is reduced to the study of all such spaces ${\bf HI}_k(\Delta,\Delta^p;n;s_1,\ldots,s_p)$ for integers $p,k,n,s_1,\ldots,s_p$ satisfying $p\ge 2$, $1\le k\le p$, and the identities proved by Ng \cite{Ng10} in the above.
In \cite{Ch16c}, Chan showed that some spaces ${\bf HI}_k(\Delta,\Delta^p;n;s_1,\ldots,s_p)$ could be empty under the settings. Thus, one of the difficulties in the study of ${\bf HI}(\Delta,\Delta^p)$ for general $p\ge 2$ is that in general we do not even know the existence of holomorphic isometries with some prescribed sheeting numbers.

In this direction, we have the following rigidity result for the $p$-th root embedding, which is the minimal case when we consider the global sheeting number and the normalizing constant $k$ equals one.

\begin{thm}[cf. Ng \cite{Ng10} for $p=2$ and $p\ge 3$ being odd; Chan \cite{Ch16a} for any $p\ge 2$]\label{GR_pth root}
Let $p\ge 2$ be an integer.
If $f:(\Delta,ds_\Delta^2)\to(\Delta^p,ds_{\Delta^p}^2)$ is a holomorphic isometric embedding with the global sheeting number $n$ equal to $p$, then $f$ is the $p$-th root embedding $F_p$ up to reparametrizations, i.e., up to automorphisms of $\Delta$ and $\Delta^p$.
\end{thm}
\begin{rem}
In the sense of Mok \cite{Mok11}, this theorem shows that the $p$-th root embedding $F_p$ is globally rigid in the space ${\bf HI}_1(\Delta,\Delta^p;p)$.
\end{rem}

To prove Theorem \ref{GR_pth root}, Chan \cite{Ch16a} made use of the polarized functional equations for any holomorphic isometry $f\in {\bf HI}_1(\Delta,\Delta^p;p)$ with $f(0)={\bf 0}$ and certain special boundary behaviour of $f$ around the branch points located on the unit circle $\partial\Delta$ together with the Puiseux series expansion around each of these branch points. This yields some restriction on the structure of the $p$-sheeted branched covering $\pi:V\to \mathbb P^1$.

\medskip
For the maximal case when considering the global sheeting number, we have the following classification result by Ng \cite{Ng10}.

\begin{thm}[cf. Theorem 6.8 in Ng \cite{Ng10}]\label{Thm_MaxGSN}
Let $F\in {\bf HI}_1(\Delta,\Delta^p;n;s_1,\ldots,s_p)$ be a holomorphic isometry.
If $\max\{s_j\mid 1\le j\le p\} = 2^{p-1}$, then $F$ can be factorized as 
$F=G_{p-1}\circ G_{p-2}\circ \cdots \circ G_1$, where $G_j:\Delta^j\to \Delta^{j+1}$ is given by $G_j(z_1,\ldots,z_j)=(z_1,\ldots,z_{j-1},\alpha_j(z_j),\beta_j(z_j))$ for $j=2,\ldots,p-1$ and $G_1(z_1)=(\alpha_1(z_1),\beta_1(z_1))$ for some $(\alpha_j,\beta_j)\in {\bf HI}_1(\Delta,\Delta^2)$, $j=1,\ldots,p-1$. 
\end{thm}

\subsection{The complete classification of holomorphic isometries from the Poincar\'e disk into the $p$-disk for $p$ at most four}
In \cite{Mok11}, Mok raised a question that whether all holomorphic isometries in ${\bf HI}(\Delta,\Delta^p)$ are obtained from the $q$-th root embeddings $F_q$, diagonal embeddings $\varpi_r$, and automorphisms of the unit disk $\Delta$ and the $p$-disk $\Delta^p$ for any $p\ge 2$ by the constructions mentioned at the beginning of the present article.
We call this question the classification problem of holomorphic isometries in ${\bf HI}(\Delta,\Delta^p)$.
In this direction, Ng \cite{Ng10} obtained the following classification result for all holomorphic isometries in ${\bf HI}(\Delta,\Delta^p)$ when $p=2,3$ or the normalizing constant $k$ equals $p$.
\begin{thm}[cf. Ng \cite{Ng10}] \label{Thm_Class_2or3}
Let $f\in {\bf HI}_k(\Delta,\Delta^p)$ be a holomorphic isometry such that all component functions of $f$ are non-constant functions.
Then, we have the following:
\begin{enumerate}
\item If $p=2$ and $k=1$, then $f$ is the square root embedding $F_2$ up to reparametrizations.
\item If $p=3$ and $k=1$, then up to reparametrizations, either $f$ is the cube root embedding $F_3$ or $f=(\alpha_1,\alpha_2\circ \beta_1,\beta_2\circ \beta_1)$ for some $(\alpha_j,\beta_j)\in {\bf HI}_1(\Delta,\Delta^2)$, $j=1,2$.
\item If $p=3$ and $k=2$, then $f(z)=(z,\alpha(z),\beta(z))$ up to reparametrizations, where $(\alpha,\beta)\in {\bf HI}_1(\Delta,\Delta^2)$.
\item If $k=p$, then $f(z)=(z,\ldots,z)$ is the diagonal embedding $\varpi_p$ up to reparametrizations.
\end{enumerate}
\end{thm}

Motivated by the classifcation problem raised by Mok \cite{Mok11} and the classification results by Ng \cite{Ng10}, Chan \cite{Ch16b} classified all holomorphic isometries in ${\bf HI}(\Delta,\Delta^4)$ as follows:
\begin{thm}[cf. Chan \cite{Ch16b}]\label{Thm_Class_4}
Let $f\in {\bf HI}_k(\Delta,\Delta^4)$ be a holomorphic isometry such that all component functions of $f$ are non-constant functions.
\begin{enumerate}
\item
If $k=1$, then $f$ is one of the following up to reparametrizations:
\begin{enumerate}
\item the $4$-th root embedding $F_4:\Delta\to \Delta^4$,
\item $\left(\alpha_1,\,\alpha_2 \circ \beta_1,\, \alpha_3\circ(\beta_2\circ\beta_1),\,\beta_3\circ(\beta_2\circ\beta_1)\right)$,
where $(\alpha_j,\beta_j)$ $\in$ ${\bf HI}_1(\Delta,\Delta^2;2)$ for $j=1,2,3$,
\item $(\alpha_1,h^2\circ \alpha_2,h^3\circ \alpha_2,h^4\circ \alpha_2)$, where $(\alpha_1,\alpha_2) \in {\bf HI}_1(\Delta,\Delta^2;2)$ and
$(h^2,h^3,h^4)\in {\bf HI}_1(\Delta,\Delta^3;3)$,
\item $(\beta_1,\alpha_1 \circ \beta_2, \alpha_2\circ \beta_2, \beta_3)$, where $(\beta_1,\beta_2,\beta_3)\in {\bf HI}_1(\Delta,\Delta^3;3)$ and $(\alpha_1,\alpha_2)\in {\bf HI}_1(\Delta,\Delta^2;2)$,
\item $(\alpha_1\circ \alpha_2,\beta_1\circ \alpha_2,\alpha_3\circ \beta_2,\beta_3\circ \beta_2)$, where $(\alpha_j,\beta_j)\in {\bf HI}_1(\Delta,\Delta^2;2)$ for $j=1,2,3$.
\end{enumerate}
\item
If $k=2$, then $f(z)$ is one of the following up to reparametrizations:
\begin{enumerate}
\item $(\alpha_1(z),\beta_1(z),\alpha_2(z),\beta_2(z))$, where $(\alpha_j,\beta_j)\in{\bf HI}_1(\Delta,\Delta^2;2)$ for $j=1,2$.
\item $(z,\alpha_1(z),(\alpha_2\circ \beta_1)(z),(\beta_2\circ \beta_1)(z))$, where $(\alpha_j,\beta_j)\in{\bf HI}_1(\Delta,\Delta^2;2)$ for $j=1,2$.
\item $(z,\alpha_1(z),\alpha_2(z),\alpha_3(z))$,
where $(\alpha_1,\alpha_2,\alpha_3)\in {\bf HI}_1(\Delta,\Delta^3;3)$.
\end{enumerate}
\item If $k=3$, then 
\[ f(z)=(z,z,\alpha(z),\beta(z)) \]
up to reparametrizations, where $(\alpha,\beta)\in {\bf HI}_1(\Delta,\Delta^2;2)$.
\item If $k=4$, then $f(z)=(z,z,z,z)$ is the diagonal embedding $\varpi_4$ up to reparametrizations.
\end{enumerate}
\end{thm}
\begin{rem}\label{Remark2.6}
\begin{enumerate}
\item[(1)] This theorem solves the classification problem in the affirmative when the target is the $4$-disk.
\item[(2)] Actually, item $(4)$ of this theorem follows from a result of Ng \cite{Ng10} (see item $(4)$ of Theorem \ref{Thm_Class_2or3}) and item $(3)$ has been proved by using the method of Ng \cite{Ng10}.
\end{enumerate}
\end{rem}

In \cite{Ch16b}, Chan listed all possible sheeting numbers and found examples obtained from the $q$-th root embeddings and the diagonal embeddings for the prescribed sheeting numbers.
One may observe that some cases can be done by using the method of Ng \cite{Ng10} in general when the normalizing constant $k$ is at most two. But for certain prescribed sheeting numbers, one could not apply Ng's method directly. 

When the normalizing constant $k$ equals one, there are two difficult cases which cannot be handled by using Ng's method directly.
In \cite{Ch16b}, Chan made use of some techniques as in the proof of the global rigidity of the $p$-th root embedding to handle one of the cases, namely, the space ${\bf HI}_1(\Delta,\Delta^4;6;3,3,6,6)$. Actually, Chan \cite[Proposition 2.4.]{Ch16b} established a general theorem which implies that for any holomorphic isometry $f\in{\bf HI}_1(\Delta,\Delta^p)$, $p\ge 3$, such that there are two distinct component functions of $f$ which have sheeting numbers equal to three, we have $f=(\alpha,\beta,F\circ \gamma)$ up to reparametrizations, where $F\in {\bf HI}_1(\Delta,\Delta^{p-2})$ and $(\alpha,\beta,\gamma)=F_3$ is the cube root embedding up to reparametrizations. 
The most difficult case is the space ${\bf HI}_1(\Delta,\Delta^4;8;4,4,4,4)$. In \cite{Ch16b}, Chan gave a criterion for reducing the classification problem of ${\bf HI}_k(\Delta,\Delta^p)$ to that of ${\bf HI}_k(\Delta,\Delta^{p-1})$ for $1\le k\le p-1$. Then, the rest is to prove that any $f\in {\bf HI}_1(\Delta,\Delta^4;8;4,4,4,4)$ satisfies such a criterion. In this direction, Chan \cite{Ch16b} made use of the branching behaviour of component functions of $f$, the polarized functional equations together with the finite branched coverings as meromorphic functions on compact Riemann surfaces and comparison of zeros and poles of certain meromorphic functions on compact Riemann surfaces.

When the normalizing constant $k$ equals two, we still need some arguments for classifying all holomorphic isometries in the space ${\bf HI}_2(\Delta,\Delta^4;n;2,2,2,2)$ by using both the method in the study of ${\bf HI}_1(\Delta,\Delta^4)$ by Chan \cite{Ch16b} and the method of Ng \cite{Ng10}.

\subsection{Classification of holomorphic isometries with special prescribed sheeting numbers or normalizing constants}
In \cite{Ch16b}, Chan generalized the methods in the proof of the global rigidity of the $p$-th root embeddings for $p\ge 2$ and obtained a generalization as follows:

\begin{thm}\label{Thm_GeneralGR}
Let $p$ and $k$ be integers satisfying $p\ge 2$,
$1\le k\le p$, ${p\over k}\in \mathbb Z$ and ${p\over k}\ge 2$.
We also let $f=(f^1,\ldots,f^p):(\Delta,kds_\Delta^2)\to (\Delta^p,ds_{\Delta^p}^2)$ be a holomorphic isometric embedding, where the global sheeting number is equal to $q:={p\over k}$.
Then, $f=(g_1,\ldots,g_k)$ up to reparametrizations, where $g_j=F_q$ up to reparametrizations for $1\le j\le k$ such that branching loci of all $g_j$'s are the same and $F_q:\Delta\to \Delta^q$ is the $q$-th root embedding.
\end{thm}
\begin{rem}
This classifies all holomorphic isometries in ${\bf HI}_k(\Delta,\Delta^p;{p\over k})$ when $k|p$ and $p\ge 2$, which is the minimal case for general normalizing constant $k$ under the assumption that $k|p$.
(Noting that Chan \cite{Ch16c} observed that ${\bf HI}_k(\Delta,\Delta^p;\lceil{p\over k}\rceil)$ could be empty for some integers $k$ and $p\ge 2$ such that $1\le k\le p$ and $k$ does not divide $p$.)
\end{rem}

On the other hand, Chan \cite{Ch16c} proved the following two propositions for certain prescribed sheeting numbers as analogues of Theorem \ref{Thm_MaxGSN}. Actually, the following two propositions follow from induction and the general theorem mentioned in item (3) of Remark \ref{Remark2.6} (cf. Chan \cite[Proposition 2.4.]{Ch16b}).
\begin{pro}[cf. Chan \cite{Ch16b, Ch16c}]\label{Pro_SSN_3power_1}
Let
\[ f=(f^1,\ldots,f^{2q+1})\in {\bf HI}_1(\Delta,\Delta^{2q+1};n;3,3,3^2,3^2,\ldots,3^{q-1},3^{q-1},3^q,3^q,3^q), \]
where $q\ge 2$ and $n\ge 1$ are some integers.
Then, $f=G_{q}\circ G_{q-1}\circ \cdots \circ G_1$ up to reparametrizations, where
$G_j :\Delta^{2j-1} \to \Delta^{2j+1}$ is given by
\[G_j(z_1,\ldots,z_{2j-1}) = (z_1,\ldots,z_{2j-2},\alpha_j(z_{2j-1}),\beta_j(z_{2j-1}),\gamma_j(z_{2j-1}))\]
for $2\le j\le q$ and $G_1=(\alpha_1,\beta_1,\gamma_1)$ for some $(\alpha_j,\beta_j,\gamma_j)\in {\bf HI}_1(\Delta,\Delta^3;3)$, $1\le j\le q$.
\end{pro}
\begin{pro}[cf. Chan \cite{Ch16b, Ch16c}]\label{Pro_SSN_3power_2}
Let
\[ f=(f^1,\ldots,f^{2q'+2})\in {\bf HI}_1(\Delta,\Delta^{2q'+2};n';3,3,3^2,3^2,\ldots,3^{q'},3^{q'},2\cdot 3^{q'},2\cdot 3^{q'}), \]
where $q'\ge 2$ and $n'\ge 1$ are some integers.
Then, we have
$f = G \circ G_{q'}\circ \cdots \circ G_1$ up to reparametrizations,
where $G_j :\Delta^{2j-1} \to \Delta^{2j+1}$ is given by
\[ G_j(z_1,\ldots,z_{2j-1}) = (z_1,\ldots,z_{2j-2},\alpha_j(z_{2j-1}),\beta_j(z_{2j-1}),\gamma_j(z_{2j-1}))\]
for $2\le j\le q'$ (if $q'\ge 2$),
$G_1=(\alpha_1,\beta_1,\gamma_1)$
and
$G:\Delta^{2q'+1}\to \Delta^{2q'+2}$ is given by
$G(z_1,\ldots,z_{2q'+1})=(z_1,\ldots,z_{2q'},h_1(z_{2q'+1}),h_2(z_{2q'+1}))$
for some $(h_1,h_2)\in {\bf HI}_1(\Delta,\Delta^2;2)$ and $(\alpha_j,\beta_j,\gamma_j)\in {\bf HI}_1(\Delta,\Delta^3;3)$ for $1\le j\le q'$.
\end{pro}

On the other hand, Chan \cite{Ch16b, Ch16c} observed that with sufficiently large normalizing constant $k$, the classification problem for ${\bf HI}(\Delta,\Delta^p)$ would reduce to the classification problem for ${\bf HI}(\Delta,\Delta^q)$ for some $q<p$ as follows:

\begin{pro}[cf. Chan \cite{Ch16c}]\label{Lem_Pdisk_iso_const_p-l}
Let $l$ and $p$ be integers satisfying $p\ge 4$ and $1\le l < {p\over 2}$.
If $f\in {\bf HI}_{p-l}(\Delta,\Delta^p)$ such that all component functions of $f$ are non-constant, then $f$ is the map $(H_1,H_2):\Delta \to \Delta^p$ up to reparametrizations for some
$H_1\in \mathbf{HI}_{l'-l}(\Delta,\Delta^{l'})$,
$H_2\in \mathbf{HI}_{p-l'}(\Delta,\Delta^{p-l'})$
and some integer $l'$ satisfying $l+1\le l'\le 2l$.
\end{pro}

In this proposition, $H_2\in \mathbf{HI}_{p-l'}(\Delta,\Delta^{p-l'})$ is the diagonal embedding $\varpi_{p-l'}$ up to reparametrizations by \cite{Ng10}.
Moreover, Chan \cite{Ch16c} made use of this proposition to classify all holomorphic isometries in ${\bf HI}_k(\Delta,\Delta^p)$ for $k=p-1$ or $p-2$ when $p\ge 5$. (Noting that for $2\le p\le 4$ the classification result already follows from Ng \cite{Ng10} and Chan \cite{Ch16b}.)
More precisely, we have

\begin{pro}[cf. Chan \cite{Ch16c}]\label{Pro_SLNC}
Let $f\in {\bf HI}_{k}(\Delta,\Delta^p)$ be such that all component functions of $f$ are non-constant.
Then, we have the following:
\begin{enumerate}
\item If $p\ge 3$ and $k=p-1$, then $f(z)=(\alpha_1(z),\alpha_2(z),z,\ldots,z)$ up to reparametrizations, where $(\alpha_1,\alpha_2)\in {\bf HI}_1(\Delta,\Delta^2;2)$.
\item If $p\ge 5$ and $k=p-2$, then $f(z)$ is one of the following up to reparametrizations:
\begin{enumerate}
\item $(H_1(z),H_2(z))$, where $H_2\in \mathbf{HI}_{p-3}(\Delta,\Delta^{p-3})$ is given by $H_2(z)=(z,\ldots,z)$ and $H_1(z)$ is either the cube root embedding or the map $(\alpha_1(z),\alpha_2(\beta_1(z)),\beta_2(\beta_1(z)))$ for some $(\alpha_j,\beta_j)\in {\bf HI}_1(\Delta,\Delta^2)$, $j=1,2$;
\item $(G_1(z),G_2(z))$, where $G_2 \in \mathbf{HI}_{p-4}(\Delta,\Delta^{p-4})$ is given by $G_2(z)=(z,\ldots,z)$ and $G_1(z)$ is one of the following:
\begin{enumerate}
\item $(\alpha_1(z),\beta_1(z),\alpha_2(z),\beta_2(z))$ with $(\alpha_j,\beta_j)\in{\bf HI}_1(\Delta,\Delta^2;2)$ for $j=1,2$,
\item $(z,\alpha_1(z),(\alpha_2\circ \beta_1)(z),(\beta_2\circ \beta_1)(z))$ with $(\alpha_j,\beta_j)\in{\bf HI}_1(\Delta,\Delta^2;2)$ for $j=1,2$,
\item $(z,\alpha_1(z),\alpha_2(z),\alpha_3(z))$ with $(\alpha_1,\alpha_2,\alpha_3)\in {\bf HI}_1(\Delta,\Delta^3;3)$.
\end{enumerate}
\end{enumerate}
\end{enumerate}
\end{pro}

By using Proposition \ref{Pro_SLNC} and Theorem \ref{Thm_GeneralGR}, Chan \cite{Ch16c} proved the following:
\begin{thm}[cf. Chan \cite{Ch16c}]\label{Thm_Prime}
Let $p$ and $k$ be integers satisfying $p\ge 2$ and $p\ge k \ge 1$.
We also let $n$ be a prime number such that ${p\over k}\le n\le 2^{p-1}$.
Then, either ${\bf HI}_k(\Delta,\Delta^p;n)= \varnothing$ or ${\bf HI}_k(\Delta,\Delta^p;n)\neq\varnothing$ and all holomorphic isometries in ${\bf HI}_k(\Delta,\Delta^p;n)$ are parametrized by the $n$-th root embedding, the diagonal embeddings, and automorphisms of $\Delta$ and $\Delta^p$.
\end{thm}
\begin{rem}
This solves the classification problem of Mok \cite{Mok11} in the affirmative for the space ${\bf HI}_k(\Delta,\Delta^p;n)$ when the global sheeting number $n$ is a prime.
\end{rem}

In the proof of Theorem \ref{Thm_Prime}, it suffices to consider the case where 
$p$ and $k$ are integers such that $p \ge 5$ and $p-3 \ge k \ge 2$ because other cases have been done by the previous classification results and \cite{Ng10}.
Then, Chan \cite{Ch16c} showed that if $n$ is a prime number such that ${\bf HI}_k(\Delta,\Delta^p;n)\neq \varnothing$, then either (i) $k|p$ and $n={p\over k}$ or (ii) $(n-1)|(p-k)$ and $n={p-l\over k-l}\ge {p-1\over k-1}$ for some integer $l\ge 1$.
In other words, if none of the conditions (i) and (ii) hold true, then the space ${\bf HI}_k(\Delta,\Delta^p;n)$ is empty.
Conversely, Chan \cite{Ch16c} also showed that given a prime number $n$ and positive integers $k$ and $p$ such that either (i) or (ii) holds true, the space ${\bf HI}_k(\Delta,\Delta^p;n)$ is non-empty.
Then, Chan \cite{Ch16c} further obtained the classification of all holomorphic isometries in ${\bf HI}_k(\Delta,\Delta^p;n)$.
When condition (i) holds true, the classification of all holomorphic isometries in the space ${\bf HI}_k(\Delta,\Delta^p;n)$ follows from Theorem \ref{Thm_GeneralGR} (cf. \cite{Ch16c}).
In the case where condition (ii) holds true, Chan \cite{Ch16c} proved that any $f\in {\bf HI}_k(\Delta,\Delta^p;n)$ is the map
$(\varpi_l,g_1,\ldots,g_{k-l})$ up to reparametrizations,
where $\varpi_l:\Delta\to\Delta^l$ is the diagonal embedding given by $\varpi_l(z)=(z,\ldots,z)$ and $g_j:\Delta\to \Delta^n$ is the $n$-th root embedding $F_n$ up to reparametrizations, $1\le j\le k-l$, such that the branching loci of all $g_j$'s are the same for $1\le j\le k-l$.

\medskip
In general, the classification problem for ${\bf HI}(\Delta,\Delta^p)$, $p\ge 5$, is difficult by using the constructions in \cite{Ng10, Ch16c} because when $p=5$ we already find some spaces ${\bf HI}_k(\Delta,\Delta^5;n;s_1,\ldots,s_5)$ from the settings of Ng \cite{Ng10}, where $k,s_1,\ldots,s_5$ and $n$ are some positive integers satisfying $1\le k\le 5$, ${5\over k}\le n \le 16$ and $s_j|n$ for $1\le j\le 5$, such that one does not have any known example of holomorphic isometry in such spaces. On the other hand, the structures of the finite branched coverings $\pi:V\to \mathbb P^1$ and $\pi_j:V_j\to \mathbb P^1$, $j=1,\ldots,p$, are much more complicated in general for the study of ${\bf HI}_k(\Delta,\Delta^p)$ with $p\ge 5$.

\subsection{Holomorphic isometries between polydisks}
In \cite{Ng10}, Ng obtained the following classification theorem on holomorphic isometries between polydisks as mentioned at the beginning of the present article.

\begin{thm}[cf. Section 9 in Ng \cite{Ng10}]\label{Ngpoly}
Let $f:(\Delta^q,\lambda ds_{\Delta^q}^2)\to (\Delta^p,ds_{\Delta^p}^2)$ be a holomorphic isometry, where $\lambda >0$ is a real constant.
Then, $\lambda=:k$ is a positive integer satisfying $k\le {p\over q}$ and we have $f(z_1,\ldots,z_q)=(G_1(z_1),\ldots,G_q(z_q))$ up to reparametrizations for some holomorphic isometries
$G_j:(\Delta,k ds_{\Delta}^2)\to (\Delta^{p_j},ds_{\Delta^{p_j}}^2)$, $j=1,\ldots,q$, where $p_1,\ldots,p_q$ are some positive integers satisfying $\sum_{j=1}^q p_j = p$.
\end{thm}

This theorem implies that the classification problem of holomorphic isometries from $(\Delta^q,k ds_{\Delta^q}^2)$ into $(\Delta^p,ds_{\Delta^p}^2)$, $p\ge 2$, is reduced to that of holomorphic isometries from $(\Delta,k ds_{\Delta}^2)$ into $(\Delta^{p'},ds_{\Delta^{p'}}^2)$ for any positive integer $p'\le p$.

\section{On holomorphic isometries from the Poincar\'e disk into the product of complex unit balls}
\subsection{On holomorphic isometries from the Poincar\'e disk into the product of the Poincar\'e disk with the complex unit ball}
Let $\mathbb{B}^n=\{z \in \mathbb{C}^n: \lVert z\rVert <1 \}$ be the unit ball in $\mathbb{C}^n.$
In \cite{YZ12}, Yuan-Zhang proved that any holomorphic isometry from $(\mathbb B^n,g_{\mathbb B^n})$, $n\ge 2$, into $(\mathbb B^{N_1},\lambda_1 g_{\mathbb B^{N_1}})\times\cdots \times (\mathbb B^{N_m},\lambda_m g_{\mathbb B^{N_m}})$ is totally geodesic, generalizing earlier results of \cite{Mok02} and \cite{Ng11}, where $\lambda_j>0$, $1\le j\le m$, are real constants.
In the case where the domain is the Poincar\'e disk $(\Delta,g_\Delta)$, the study of all holomorphic isometries in ${\bf HI}(\Delta,\Delta^p)$ is a special case in the study of all holomorphic isometries from $(\Delta,g_\Delta)$ into $(\mathbb B^{N_1},\lambda_1 g_{\mathbb B^{N_1}})\times\cdots \times (\mathbb B^{N_m},\lambda_m g_{\mathbb B^{N_m}})$, where $\lambda_j>0$, $1\le j\le m$, are real constants.
Moreover, we have seen that there are many non-standard (i.e., not totally geodesic) holomorphic isometries in ${\bf HI}(\Delta,\Delta^p)$. Therefore, it is expected that there are many non-standard holomorphic isometries from $(\Delta,g_\Delta)$ into $(\mathbb B^{N_1},\lambda_1 g_{\mathbb B^{N_1}})\times\cdots \times (\mathbb B^{N_m},\lambda_m g_{\mathbb B^{N_m}})$ and the classification of such holomorphic isometries would be much more difficult than that of ${\bf HI}(\Delta,\Delta^p)$.

In \cite{Ng10}, Ng proved that any holomorphic isometry in ${\bf HI}_1(\Delta,\Delta^2)$ with non-constant component functions is the square root embedding up to reparametrizations.
Motivated by the study in \cite{Ng10} and \cite{YZ12}, it is natural to study all holomorphic isometries from $(\Delta,g_\Delta)$ into $(\mathbb B^{N_1}, g_{\mathbb B^{N_1}})\times (\mathbb B^{N_2}, g_{\mathbb B^{N_2}})$ for $N_1,N_2\ge 1$ and at least one of the $N_1$ and $N_2$ is greater than or equal to two.
The simplest situation in this study is when $N_1=1$ and $N_2\ge 2$ because any holomorphic isometry from $(\Delta,g_\Delta)$ into $(\Delta,g_\Delta)\times (\mathbb B^{N_2}, g_{\mathbb B^{N_2}})$ can be naturally regarded as a holomorphic isometry from $(\Delta,g_\Delta)$ into $(\mathbb B^{N_1}, g_{\mathbb B^{N_1}})\times (\mathbb B^{N_2}, g_{\mathbb B^{N_2}})$ for any integer $N_1\ge 2$.
This motivated the study in \cite{CY17} and Chan-Yuan obtained the following existence theorem:

\begin{pro}[Chan-Yuan \cite{CY17}]\label{Thm_Existence1}
Let $n\ge 2$ be an integer and $f:=(f_1,f_{2,1},\ldots,f_{2,n}):\Delta \to \Delta \times \mathbb B^n$ be a holomorphic map.
Then, $f$ is a holomorphic isometry from $(\Delta,g_\Delta)$ into $(\Delta,g_\Delta)\times (\mathbb B^n,g_{\mathbb B^n})$ with $f(0)=\bf 0$ if and only if
\begin{equation}\label{Eq:Existence1}  
{\bf U} \begin{pmatrix}
f_1(w),f_{2,1}(w),\ldots, f_{2,n}(w)
\end{pmatrix}^t
= \begin{pmatrix}
w,f_1(w)  f_{2,1}(w),\ldots, f_1(w) f_{2,n}(w)
\end{pmatrix}^t.
\end{equation}
on $\Delta$ for some unitary matrix ${\bf U}:=\begin{pmatrix}
u_{ij}
\end{pmatrix}_{1\le i,j\le n+1} \in U(n+1)$.
Moreover, given any unitary matrix ${\bf U}\in U(n+1)$, there exists a unique holomorphic isometry $(f_1,f_{2,1},\ldots,f_{2,n}):(\Delta,g_\Delta)\to (\Delta,g_\Delta)\times (\mathbb B^n,g_{\mathbb B^n})$ up to reparametrizations, such that equation (\ref{Eq:Existence1}) holds.
\end{pro}

Assuming $f(0)={\bf 0}$, the condition of $f:=(f_1,f_{2,1},\ldots,f_{2,n}):(\Delta,g_\Delta)\to(\Delta,g_\Delta)\times (\mathbb B^n,g_{\mathbb B^n})$ being a holomorphic isometry is equivalent to the condition that $f$ satisfies
the following functional equation
\[ (1-|f_1(w)|^2) \left(1-\sum_{j=1}^n |f_{2,j}(w)|^2\right) = 1-|w|^2 \]
on $\Delta$.
Then, it follows from the local rigidity theorem of Calabi \cite[Theorem 2]{Ca53} that there exists ${\bf U}\in U(n+1)$ such that
\[  {\bf U} \begin{pmatrix}
f_1(w),f_{2,1}(w),\ldots, f_{2,n}(w)
\end{pmatrix}^t
= \begin{pmatrix}
w,f_1(w)  f_{2,1}(w),\ldots, f_1(w) f_{2,n}(w)
\end{pmatrix}^t. \]
On the other hand, if equation (\ref{Eq:Existence1}) holds, then one can easily solve $f$ out as a germ of holomorphic isometry which extends to a global holomorphic isometry.

In order to reduce the complexity in the constructions of holomorphic isometries from a given unitary matrix ${\bf U}\in U(n+1)$, Chan-Yuan \cite{CY17} renormalized the matrix ${\bf U}$ by certain automorphisms of $\Delta$ and $\Delta\times \mathbb B^n$ so that one may assume that $\begin{pmatrix}
u_{ij}
\end{pmatrix}_{2\le i,j \le n+1}$ is an upper triangular matrix if we write ${\bf U}=\begin{pmatrix}
u_{ij}
\end{pmatrix}_{1\le i,j \le n+1}\in U(n+1)$. This does not change the equivalence class of the given holomorphic isometry $f:(\Delta,g_\Delta)\to (\Delta,g_\Delta)\times (\mathbb B^n,g_{\mathbb B^n})$.
Since $f_1\equiv 0$ would imply that $f$ is a totally geodesic holomorphic isometric embedding from $(\Delta,g_\Delta)$ into $(\mathbb B^n,g_{\mathbb B^n})$, from now on we may assume that $f_1$ is a non-constant function.
In fact, Chan-Yuan \cite{CY17} showed that $f_1\equiv 0$ if and only if $\det\begin{pmatrix}
u_{ij}
\end{pmatrix}_{2\le i,j\le n+1}=0$ under the above settings.
In particular, in what follows we may assume that $\begin{pmatrix}
u_{ij}
\end{pmatrix}_{2\le i,j\le n+1}$ is invertible.
Then, from the constructions we have $f_{2,j}(w) = R_j(f_1(w))$ for some rational function $R_j:\mathbb P^1\to \mathbb P^1$, $1\le j\le n$, and there is a rational function $R:\mathbb P^1\to \mathbb P^1$ such that $R(f_1(w))=w$ for all $w\in \Delta$.
(Noting that $R_j$ and $R$ can be written explicitly in terms of the entries of the unitary matrix ${\bf U}$.)

Under the above settings, as a key step, Chan-Yuan \cite{CY17} showed that the rational function $R:\mathbb P^1\to \mathbb P^1$ satisfies $R\left({1\over \overline z}\right)={1\over \overline{R(z)}}$ and
\[ R(z)=\alpha_0 z\prod_{j=1}^n {z-{1\over \overline{\alpha_j}}\over z-\alpha_j}, \]
for some $\alpha_j\in \overline\Delta \smallsetminus \{0\}$, $1\le j\le n$, and $\alpha_0\in \overline\Delta\smallsetminus \{0\}$.
In particular, we have $\deg(R)\le n+1$.
This yields the following theorem.
\begin{thm}[Chan-Yuan \cite{CY17}]\label{Thm_Reduction1}
Let $f=(f_1,f_{2,1},\ldots,f_{2,n}): (\Delta,g_\Delta)\to (\Delta,g_\Delta)\times(\mathbb B^n,g_{\mathbb B^n})$ be a holomorphic isometry with $f(0)=\bf 0$ such that $f_1$ is non-constant, where $n\ge 2$ is an integer.
Let $R:\mathbb P^1\to \mathbb P^1$ be the rational function such that $R(f_1(w))=w$.
Then, we have the following:
\begin{enumerate}
\item[(1)] If $\deg(R)=m+1\le n$ for some positive integer $m$, then $f$ is congruent to a map $\widetilde f: (\Delta,g_\Delta)\to (\Delta,g_\Delta)\times(\mathbb B^n,g_{\mathbb B^n})$ defined by $\widetilde f(w)=(F(w),{\bf 0})$ for some holomorphic isometry $F:(\Delta,g_\Delta)\to (\Delta,g_\Delta)\times(\mathbb B^{m},g_{\mathbb B^{m}})$ such that $F$ is incongruent to
$(\widetilde F,{\bf 0})$ for any holomorphic isometry $\widetilde F:(\Delta,g_\Delta)\to (\Delta,g_\Delta)\times(\mathbb B^{m'},g_{\mathbb B^{m'}})$ with $m'\le m-1$ (resp. $\widetilde F:(\Delta,g_\Delta)\to (\Delta,g_\Delta)$)
whenever $m\ge 2$ (resp. $m=1$).
\item[(2)]
If $\deg(R)=1$, then $f$ is congruent to a map $\widetilde f: (\Delta,g_\Delta)\to (\Delta,g_\Delta)\times(\mathbb B^n,g_{\mathbb B^n})$ defined by $\widetilde f(w)=(w,{\bf 0})$.
\end{enumerate}
\end{thm}

The theorem simply transforms the equivalence relation of two holomorphic isometries from $ (\Delta,g_\Delta)$ to $(\Delta,g_\Delta)\times(\mathbb B^n,g_{\mathbb B^n})$ to the equivalence relation of the rational function $R(z)$ and further uses the degree of $R(z)$ to reduce the dimension of the complex unit ball $\mathbb B^n$ in the target $\Delta\times \mathbb B^n$. Moreover, Chan-Yuan \cite{CY17} showed the existence of holomorphic isometry from $(\Delta,g_\Delta)$ into $(\Delta,g_\Delta)\times(\mathbb B^n,g_{\mathbb B^n})$, $n\ge 2$, in which the degree of the corresponding rational function $R$ is equal to $n+1$. 
From the previous results and constructions, Chan-Yuan \cite{CY17} obtained the following theorem so as to describe the space of all holomorphic isometries from $(\Delta,g_\Delta)$ into $(\Delta,g_\Delta)\times(\mathbb B^n,g_{\mathbb B^n})$ in an explicit way through $R(z)$.
\begin{thm}[Chan-Yuan \cite{CY17}]\label{Thm:Red_Para1}
Let $f=(f_1,f_{2,1},\ldots,f_{2,n}): (\Delta,g_\Delta)\to (\Delta,g_\Delta)\times(\mathbb B^n,g_{\mathbb B^n})$ be a holomorphic isometry with $f(0)=\bf 0$ such that $f_1$ is non-constant, where $n\ge 2$ is an integer.
Then, $f$ can be determined by some holomorphic isometry $\widetilde f=(\widetilde f_1,\widetilde f_{2,1},\ldots,\widetilde f_{2,n-1}):(\Delta,g_\Delta)\to (\Delta,g_\Delta)\times(\mathbb B^{n-1},g_{\mathbb B^{n-1}})$ and by some parameter $\zeta\in \overline\Delta\smallsetminus\{0\}$ up to congruence in the following explicit way: 
$$R(z) = \widetilde R(z) {\overline{\zeta}z-1\over z- \zeta},$$
where $R,\widetilde R:\mathbb P^1 \to \mathbb P^1$ are the rational functions such that $R(f_1(w))=w$ and $\widetilde R(\widetilde f_1(w))=w$ on $\Delta$.
\end{thm}




Chan-Yuan \cite{CY17} further showed the existence of a real $1$-parameter family of mutually incongruent holomorphic isometries from $(\Delta,g_\Delta)$ into $(\Delta,g_\Delta)\times(\mathbb B^n,g_{\mathbb B^n})$, where $n\ge 2$ is an integer.

\begin{pro}[Chan-Yuan \cite{CY17}]\label{Pro:Ex_1PFICHIndim}
Let $n\ge 2$ be an integer.
Then, there is a real $1$-parameter family $\{f_t\}_{t\in \mathbb R}$ of mutually incongruent holomorphic isometries $f_{t}:(\Delta,g_\Delta)\to (\Delta,g_\Delta)\times (\mathbb B^n,g_{\mathbb B^n})$.
More generally, there is a family $\{f_\zeta\}_{\zeta\in A_n}$ of holomorphic isometries $f_{\zeta}:(\Delta,g_\Delta)\to (\Delta,g_\Delta)\times (\mathbb B^n,g_{\mathbb B^n})$ such that for any $\zeta,\zeta'\in A_n:=\left\{\xi\in \mathbb C\mid {n-1\over n+1}<|\xi|<1\right\}$, $f_\zeta$ and $f_{\zeta'}$ are congruent to each other if and only if $|\zeta|=|\zeta'|$.
In addition, if $n\ge 3$, then there is a family $\{f_\zeta\}_{\zeta\in \Delta^*}$ of holomorphic isometries $f_{\zeta}:(\Delta,g_\Delta)\to (\Delta,g_\Delta)\times (\mathbb B^n,g_{\mathbb B^n})$ such that for any $\zeta,\zeta'\in \Delta^*:=\Delta\smallsetminus\{0\}$, $f_\zeta$ and $f_{\zeta'}$ are congruent to each other if and only if $|\zeta|=|\zeta'|$.
\end{pro}
\begin{rem}
The proof of this proposition in \cite{CY17} involves some explicit construction of certain unitary matrix ${\bf U}_\zeta\in U(n+1)$ depending on $\zeta$ and this would give the desired family of holomorphic isometries $f_{\zeta}$.
\end{rem}
This also yields the following interesting corollary. We will come back to this corollary in a more explicit form later.
\begin{cor}[Chan-Yuan \cite{CY17}]\label{Cor:ExHolo_n=2}
Let $f_\zeta=(f_\zeta^1,f_{\zeta}^{2,1},f_{\zeta}^{2,2}):(\Delta,g_\Delta)\to (\Delta,g_\Delta)\times (\mathbb B^2,g_{\mathbb B^2})$ be the holomorphic isometry obtained in Proposition \ref{Pro:Ex_1PFICHIndim}.
Then, for $0<|\zeta|<{1\over 3}$, $f_\zeta$ is non-standard and extends holomorphically to a neighborhood of the closed unit disk $\overline\Delta$  with $f^1_\zeta(w), f^{2,1}_\zeta(w), f_{\zeta}^{2,2}(w)$ all irrational functions in $w\in \Delta\subset \mathbb C$.
\end{cor}

In \cite{CY17}, there are some applications to the study of holomorphic isometries from $(\Delta,g_\Delta)$ into $(\Omega,g_\Omega)$ for any irreducible bounded symmetric domain $\Omega\Subset \mathbb C^N$ in its Harish-Chandra realization of rank $\ge 2$. Actually, from Corollary \ref{Cor:ExHolo_n=2} and Proposition \ref{Pro:Ex_1PFICHIndim}, Chan-Yuan \cite{CY17} constructed many new examples of holomorphic isometries from $(\Delta,g_\Delta)$ into $(\Omega,g_\Omega)$ which have the same properties as the map $f_\zeta$ in Corollary \ref{Cor:ExHolo_n=2} for any irreducible bounded symmetric domain $\Omega\Subset \mathbb C^N$ in its Harish-Chandra realization of rank $\ge 2$ except for the type-$\mathrm{IV}$ domain (cf. \cite[Theorem 3.21]{CY17}).
On the other hand, in the case where the target $\Omega$ is an irreducible bounded symmetric domain of classical type except for the type-$\mathrm{IV}$ domain, Chan-Yuan \cite{CY17} showed the existence of a real $1$-parameter family $\{f_t\}_{t\in \mathbb R}$ of mutually incongruent holomorphic isometries $f_{t}:(\Delta,g_\Delta)\to (\Omega,g_\Omega)$ by using the constructions in Proposition \ref{Pro:Ex_1PFICHIndim} and some extra constructions of holomorphic isometric embeddings from $\Delta\times \mathbb B^n$ into $\Omega$ for some $n\ge 2$.

\subsection{Rigidity of rational holomorphic isometries from the Poincar\'e disk into the product of complex unit balls}
It is natural to ask which condition would force total geodesy of a holomorphic isometry from the Poincar\'e disk into the product of complex unit balls.
For the case where the target is the $p$-disk, Theorem 3 in Mok-Ng \cite[p.\;2637]{MN09} asserts that if $f\in {\bf HI}_k(\Delta,\Delta^p)$ which extends holomorphically to some neighborhood of the closed unit disk $\overline\Delta$, then $f$ is totally geodesic.
However, Chan-Yuan \cite{CY17} observed that Corollary \ref{Cor:ExHolo_n=2} yields an example of non-standard holomorphic isometry from $(\Delta,g_\Delta)$ into $(\Delta,g_\Delta)\times (\mathbb B^n,g_{\mathbb B^n})$, $n\ge 2$, which extends holomorphically to a neighborhood of the closed unit disk $\overline\Delta$.
In particular, this suggests that one should impose some stronger assumptions in order to generalize Theorem 3 in \cite[p.\;2637]{MN09} and obtain the rigidity of a certain class of holomorphic isometries from the Poincar\'e disk into the product of complex unit balls.
More precisely, Chan-Yuan \cite{CY17} obtained the following rigidity theorem for rational holomorphic isometries from the Poincar\'e disk into the product of complex unit balls.
\begin{thm}[Chan-Yuan \cite{CY17}]\label{Thm_ExtHolo}
Let $f=(f_1,\ldots,f_m):(\Delta,g_\Delta) \to (\mathbb B^{N_1},\lambda_1g_{\mathbb B^{N_1}}) \times \cdots \times (\mathbb B^{N_m},\lambda_mg_{\mathbb B^{N_m}})$ be a holomorphic isometry
such that $f_j$ is a non-constant map for $1\le j\le m$,
where $\lambda_j$, $1\le j\le m$, are positive real constants. 
If $f$ is rational, i.e., each component function of $f$ is a rational function in $w\in \Delta\subset \mathbb C$, then $f_j:(\Delta,g_\Delta)\to (\mathbb B^{N_j},g_{\mathbb B^{N_j}})$ is a (totally geodesic) holomorphic isometry for $1\le j\le m$ and $\sum_{j=1}^m\lambda_j=1$ so that $f$ is totally geodesic.
\end{thm}

\section{Proper holomorphic map from the unit disk to the complex unit ball}
We first recall some facts about proper holomorphic maps between complex unit balls $\mathbb{B}^n$ and $\mathbb{B}^N$. It is well-known that any proper holomorphic map from the unit disk $\Delta \subset \mathbb{C}$ to itself must be a finite Blaschke product. By a classical result of Alexander \cite{Al74}, any proper holomorphic self-map of $\mathbb{B}^n, n \geq 2,$ is a holomorphic automorphism.  Moreover, Forstneric \cite{Fo89} proved that a proper holomorphic
map from $\mathbb{B}^n$ to $\mathbb{B}^N, N \geq n \geq 2$ must be rational if it extends smoothly to the unit sphere. The following results reveal a different phenomenon of proper maps from the unit disk to balls. In \cite{DHX16}, D'Angelo-Huo-Xiao proved the following result. 

\begin{pro}[Proposition 2.1 in \cite{DHX16}]
For each $N \geq 2,$ there is a non-algebraic proper holomorphic map from $\Delta$ to $\mathbb{B}^N$ that extends holomorphically to a neighborhood of the closed unit disk.
\end{pro}

We observed the following result from Corollary \ref{Cor:ExHolo_n=2} and we give a complete description here.
The existence of such examples seems to be unknown to the best knowledge of the authors'.

\begin{pro}
For each $N \geq 2,$ there is an algebraic and non-rational proper holomorphic map from $\Delta$ to $\mathbb{B}^N$ that extends holomorphically to a neighborhood of the closed unit disk $\overline{\Delta}$.
\end{pro}

\begin{proof}
Define
\[ {\bf U}_{\zeta}:=\begin{pmatrix}
-\overline{\zeta}^2 & -\sqrt{1-|\zeta|^2} & \overline{\zeta} \sqrt{1-|\zeta|^2} \\
- \sqrt{1-|\zeta|^2} \cdot \overline{\zeta} & \zeta & 1-|\zeta|^2 \\
\sqrt{1-|\zeta|^2} & 0 & \zeta
\end{pmatrix}\in U(3) \] for $0 \leq |\zeta| < 1/3$ and let $f_\zeta = (f_\zeta^1, f^{2, 1}_\zeta, f^{2, 2}_\zeta)$ be the holomorphic isometry satisfying $${\bf U_\zeta} \begin{pmatrix}
f^1_\zeta(w),f^{2,1}_\zeta(w), f^{2,2}_\zeta(w)
\end{pmatrix}^t
= \begin{pmatrix}
w,f^1_\zeta(w)  f^{2,1}_\zeta(w), f^1_\zeta(w) f^{2,2}_\zeta(w)
\end{pmatrix}^t,$$ or equivalently $$(1-|f^1_\zeta(w)|^2) \left(1-\sum_{j=1}^2 |f^{2,j}_\zeta(w)|^2\right) = 1-|w|^2.
$$
The existence and uniqueness of $f_\zeta$ is guaranteed by Proposition \ref{Thm_Existence1}. It then follows from Corollary \ref{Cor:ExHolo_n=2} that $f_\zeta$ is a non-standard, proper holomorphic map from $\Delta$ to $\Delta \times \mathbb{B}^2$ with $f_\zeta^1, f_\zeta^{2, 1}, f_\zeta^{2, 2}$ irrational and $f_\zeta$ extends
holomorphically to a neighborhood of the closed unit disk. On the other hand, there exists a nonempty open arc $A \subset \partial\Delta$ such that $f^1_\zeta(A) \subset \partial\Delta$ or $(f^{2, 1}_\zeta, f^{2, 2}_\zeta)(A) \subset \partial\mathbb{B}^2$. If the former holds, then by the identity theorem of analytic functions, $f^1_\zeta (\partial\Delta) \subset \partial\Delta$. Therefore $f^1_\zeta$ must be a finite Blaschke product as it is a proper holomorphic self map of $\Delta$. This is a contradiction as $f^1_\zeta$ is irrational according to the construction. This implies $(f^{2, 1}_\zeta, f^{2, 2}_\zeta)(A) \subset \partial\mathbb{B}^2$. By the the identity theorem of analytic functions again, $(f^{2, 1}_\zeta, f^{2, 2}_\zeta)$ is a proper holomorphic map from $\Delta$ to $\mathbb{B}^2$. It is the desired example when $N=2$. For $N\ge 2$, by composing $(f^{2, 1}_\zeta, f^{2, 2}_\zeta)$ with the standard embedding $\mathbb B^2 \hookrightarrow \mathbb B^N$, $(z_1,z_2)\mapsto (z_1,z_2,0,\ldots,0)$, we obtain the desired proper holomorphic map from $\Delta$ to $\mathbb{B}^N$.
\end{proof}

\section{Miscellaneous results on holomorphic isometries}

We provide some new results on holomorphic isometries in this section. 
First recall some facts about the Bergman kernels of irreducible bounded symmetric spaces.
Irreducible bounded symmetric domains  are classified into four types of Cartan's classical domains and two exceptional cases. Let $D$ be an irreducible bounded symmetric domain of any type in $\mathbb{C}^n, n \geq 2$. Denote by $K_D$ the Bergman kernel of $D.$ Then $K_D$ always has the following form (after a linear change of coordinates for type III):
$$K_D(Z,\overline Z)=c_D \left(1- \lVert Z\rVert^2+ Q_D(Z, \overline{Z})\right)^{m}.$$
Here $c_D \in \mathbb{R}, m \in \mathbb{Z}$ are  constants depending on the type and the dimension $n$ with $m<0$.
Here $Q_D(Z, \overline{Z})$ is a real polynomial in $Z, \overline{Z}$ with $Q_D(Z, \overline{Z})=o(\lVert Z\rVert^2).$ More precisely, $Q_D$ consists only of terms of the form $Z^{\alpha} \overline{Z}^{\beta}$ with $|\alpha| \geq 2, |\beta| \geq 2.$ For more details on the Bergman kernels of bounded symmetric domains, see \cite{H79}, \cite{Xu07}, \cite{XY16b}, etc. To make it more concrete to the readers, we give the following examples of type I and type IV.

\begin{eg}
Let $q \geq p$. Let $I_p$ be the $p \times p$ identity matrix.
The type I domain is defined by
$$D^I_{p, q} = \{Z \in M(p, q; \mathbb{C}) | I_p -  Z \overline{Z}^t >0 \}$$ and
the Bergman kernel is given by 
$$K_{D^I_{p, q}}(Z, \bar Z) = c_I \left(\det(I_p -  Z \overline{Z}^t) \right)^{-(p+q)},$$
for some positive constant $c_I$ depending on $p,q.$ We denote by $Z(\begin{matrix}
 i_{1} & ... & i_{k} \\
 j_{1} & ... & j_{k}
   \end{matrix}
)$ the determinant of the submatrix of $Z$ formed by its $i_{1}^{\text{th}},...,i_{k}^{\text{th}}$ rows and $j_{1}^{\text{th}},...,j_{k}^{\text{th}}$ columns. Then
\begin{equation}\notag
\mathrm{det}( I_{p}- Z\overline{Z^t})= 1 + \sum_{k=1}^p (-1)^k\left( \sum_{1 \leq i_{1}<i_{2}<...< i_{k}\leq p,\\ 1 \leq j_{1}< j_{2}<...< j_{k} \leq q}\left|
 Z(\begin{matrix}
 i_{1} & ... & i_{k} \\
 j_{1} & ... & j_{k}
   \end{matrix}
)\right|^2\right).
\end{equation}
\end{eg}

\begin{eg}
The type IV domain is defined by 
$$D^{IV}_n = \left\{Z =(z_1, \cdots, z_n) \in \mathbb{C}^n \;\middle|\; Z \overline{Z}^t <2 ~\text{and} ~ 1- Z \overline{Z}^t + \frac{1}{4} |Z Z^t|^2 >0 \right\}$$
 and the Bergman kernel is given by
$$K_{D^{IV}_n}(Z, \bar Z) = c_{IV} \left( 1 -  Z\overline{Z}^t +\frac{1}{4} |Z Z^t|^2 \right)^{-n}$$
for some positive constant $c_{IV}$ depending on $n.$

\end{eg}

In general when $D$ is reducible, write $D=D_1 \times \cdots \times D_k$, where each $D_j$ is an  irreducible bounded symmetric domain in $\mathbb{C}^{n_j}$. $D$ is equipped with the metric $\oplus_{i=1}^k \lambda_i g_{D_i},$ with $\lambda_j >0,$ written as
$D=(D_1, \lambda_1 g_{D_1}) \times \cdots \times (D_k, \lambda_k g_{D_k}).$ Let $\Omega=(\Omega_1, \mu_1 g_{\Omega_1}) \times \cdots \times (\Omega_l, \mu_l g_{\Omega_l}),$ with $\mu_i > 0, $  be another bounded symmetric domain, where each $\Omega_i$ is an irreducible bounded symmetric domain in $\mathbb{C}^{N_i}.$ Let $U$ be a connected open subset of $D, H=(H_1, \cdots, H_l)$ be a holomorphic map from $U$ to $\Omega,$ where $H_i: U \rightarrow \Omega_i.$ We say $H$ is a local holomorphic isometry or locally preserves the canonical metric if the following identity holds:
\begin{equation}\label{hi}
\sum_{i=1}^l \mu_i H_i^{*}g_{\Omega_i}=\bigoplus_{j=1}^k \lambda_j g_{D_j} ~\text{on}~U.
\end{equation}

We provide the following algebraicity result of local holomorphic isometries between bounded symmetric domains in the sense of equation (\ref{hi}). It can be regarded as a generalization of Mok's algebraicity result \cite{Mok12}. We believe that Theorem \ref{Thm:Algebraicity} is known to experts, but we still record it here for the sake of completeness.

\begin{thm}\label{Thm:Algebraicity}
Let $H$ be a holomorphic map from $U \subset D$ to $\Omega$ satisfying equation (\ref{hi}). Then $H$ is algebraic and extends to a proper holomorphic isometric immersion from $D$ to $\Omega.$
\end{thm}

\begin{proof}
For each irreducible factor $D_j,$ as discussed above,  we can write the K\"ahler form $\omega_{D_i}$ associated to the K\"ahler metric $g_{D_i}$ as
$$\omega_{D_j}=-\sqrt{-1} \partial \overline{\partial}\mathrm{log} \left(1-\lVert Z_j\rVert^2+Q_{j}(Z_j, \overline{Z_j}) \right).$$
Here $Z_j$ is the coordinates in $\mathbb{C}^{n_j}$ and $Q_j$ is a real polynomial with the property mentioned above: It consists only of terms of the form $Z_j^{\alpha}\overline{Z_j}^{\beta}$ with $|\alpha| \geq 2, |\beta| \geq 2.$
Similarly, write 
$$\omega_{\Omega_i}=-\sqrt{-1} \partial \overline{\partial}\mathrm{log} \left(1-\lVert \xi_i\rVert^2+P_{i}(\xi_i, \overline{\xi_i}) \right).$$
where $\xi_i$ is the coordinates in $\mathbb{C}^{N_i},$ and $P_i$ satisfies the similar condition as $Q_j$.
By composing with an automorphism of the target, we can assume $H({\bf 0})={\bf 0}.$ Then by the equation (\ref{hi}) and the standard reduction, we have the following functional equation on $U$:

\begin{equation}\label{eqnmetrich}
\prod_{i=1}^l \left(1-\lVert H_i\rVert^2+P_i(H_i, \overline{H}_i) \right)^{\mu_i}=\prod_{j=1}^k \left(1-\lVert Z_j\rVert^2+Q_j(Z_j, \overline{Z}_j) \right)^{\lambda_j}.
\end{equation}
 Write $Z=(Z_1, \cdots, Z_k)=(z_1, \cdots, z_m)$ where $m=n_1 +\cdots+ n_k.$
We can assume that $U$ is the 
ball $B({\bf 0}, r)$ in $\mathbb{C}^m$ centered at ${\bf 0}$ with radius $r$ for some small  $r >0.$ Write $x \cdot y=\sum_{i=1}^n x_i y_i$ for two vectors $x=(x_1, \cdots, x_n), y=(y_1, \cdots, y_n)$ in $\mathbb{C}^n.$ We complexify (\ref{eqnmetrich}) to get the following equation holds for $(Z, W) \in U \times U.$  Here similarly we write $W=(W_1, \cdots, W_k)=(w_1, \cdots, w_m).$

\begin{equation}\label{eqnhw}
\prod_{i=1}^l \left(1- H_i(Z) \cdot \overline{H}_i(W)+P_i(H_i(Z), \overline{H}_i(W)) \right)^{\mu_i}=\prod_{j=1}^k \left(1- Z_j \cdot W_j+Q_j(Z_j, W_j) \right)^{\lambda_j}
\end{equation}

Write $D_{\delta}=\frac{\partial}{\partial z_{\delta}}.$ For each $\alpha=(\alpha_1, \cdots, \alpha_m),$ write $D^{\alpha}=\frac{\partial}{\partial z_1^{\alpha_1} \cdots \partial z_m^{\alpha_m}}.$ We then take logarithmic differentiation to both sides of the above equation and obtain

\begin{equation}\label{eqndeltah}
\sum_{i=1}^l  \frac{\mu_i D_{\delta} (H_i) (Z) \cdot \overline{H_i} (W)+ D_{\delta} P_i(H_i(Z), \overline{H}_i(W))}{1-H_i(Z) \cdot \overline{H}_i(W)+ P_i(H_i(Z), \overline{H}_i(W))}=R_{\delta}(Z, W).
\end{equation}
Here $R_{\delta}$ is the logarithmic differentiation of the right hand side of (\ref{eqnhw}). It is rational in $Z, W.$ If we write
$$\chi=(\chi_1, \cdots, \chi_N)=(\sqrt{\mu_1} H_1, \cdots, \sqrt{\mu_l} H_l),$$
then the equation (\ref{eqndeltah}) can be expanded as
\begin{equation}\label{eqnzwr}
\sum_{i=1}^l D_{\delta}(\chi_i)(Z) \cdot \overline{\chi}(W) + \eta_{\delta}(Z, \overline{\chi}(W))= R_{\delta}(Z, W).
\end{equation}
Here $\eta_{\delta}$ is rational in $\overline{\chi}$ for the fixed $Z$. Moreover, $\eta_{\delta}$ has no constant or linear terms in its Taylor expansion with respect to $\overline{\chi}.$
Keep differentiating (\ref{eqnzwr}), we get for each $\alpha$,
$$\sum_{i=1}^l D^{\alpha}(\chi_i)(Z) \cdot \overline{\chi}(W) + \eta_{\alpha}(Z, \overline{\chi}(W))= R_{\alpha}(Z, W).$$
Here similarly $R_{\alpha}(Z, W)$ is rational in $Z, W$. Moreover, $\eta_{\alpha}$ is rational in $ \overline{\chi}$ for the fixed $Z$ and has no constant or linear terms in its Taylor expansion with respect to $\overline{\chi}.$ Then by a similar argument as in the proof of Proposition 3.1 in \cite{HY14} (See also \cite{FHX16}), we can show that $\chi(W),$ and thus $H(W)$ is algebraic in $W.$ Next since $D$ is simply connected, and $\Omega$ can be holomorphically isometrically embedded into $\mathbb{P}^{\infty},$ it follows from the classical result of Calabi \cite{Ca53} that $H$ extends holomorphically to  $D.$ By the identity theorem of analytic functions, $H$ preserves the equation (\ref{hi}) everywhere. Moreover, $H$ must be proper by comparing the asymptotic behavior of the K\"ahler metrics near the boundary. 
\end{proof}

\bigskip

We now focus on the case where $D$ and $\Omega$ are products of unit balls. This is motivated by the corresponding result of Ng in \cite{Ng10} on holomorphic isometry between polydisks (cf. Theorem \ref{Ngpoly}).
Denote by $D=(\mathbb{B}^{n_1}, \lambda_1 g_{\mathbb{B}^{n_1}}) \times \cdots \times (\mathbb{B}^{n_k}, \lambda_k g_{\mathbb{B}^{n_k}})$ the Cartesian product of the complex unit balls of possibly different dimensions equipped with the metric $\oplus_{j=1}^k \lambda_j g_{\mathbb{B}^{n_j}}.$ Let $U$ be a connected open subset of $D$. Let $F=(F_1, \cdots, F_l)$ be a holomorphic map from $U$ to $\Omega=(\mathbb{B}^{N_1}, \mu_1 g_{\mathbb{B}^{N_1}}) \times \cdots \times (\mathbb{B}^{N_l}, \mu_l g_{\mathbb{B}^{N_l}})$, where $F_i: U \rightarrow \mathbb{B}^{N_i}$.

\begin{thm}
Let $F=(F_1, \cdots, F_l): U \rightarrow \Omega$ be a holomorphic map preserving the canonical metric in the following sense:
$$\sum_{i=1}^l  \mu_i F_i^{*} g_{\mathbb{B}^{N_i}}=\bigoplus_{j=1}^k \lambda_j g_{\mathbb{B}^{n_j}}~\text{on}~U.$$
Assume that none of the $F_i'$s is constant; all $n_j, N_i$ are at least $2$ and all $\mu_i, \lambda_j$ are positive real constants. Then for each $F_i,$ there exists some $j$ such that $F_i$ only depends on the coordinates of $\mathbb{B}^{n_j}.$ Moreover, $F_i$ extends to a totally geodesic holomorphic isometric embedding from 
$(\mathbb{B}^{n_j},g_{\mathbb{B}^{n_j}})$ to $(\mathbb{B}^{N_i},g_{\mathbb{B}^{N_i}}).$ Consequently, it holds that $\sum_{i=1}^l \mu_i=\sum_{j=1}^k \lambda_j.$
\end{thm}

\begin{proof}
Note in the case of $k=1,$ the theorem was proved in \cite{YZ12}. We now assume $k \geq 2$. By composing  with an automorphism of the target, assume $F({\bf 0})={\bf 0}.$
By the metric preserving assumption and the standard reduction, we have the following functional equation:
\begin{equation}\label{eqnmetric}
\prod_{i=1}^l \left(1-\lVert F_i\rVert^2 \right)^{\mu_i}=\prod_{j=1}^k \left(1-\lVert Z_j\rVert^2 \right)^{\lambda_j},
\end{equation}
where $Z_j$ is the coordinates in $\mathbb{C}^{n_j}$.

By Theorem \ref{Thm:Algebraicity}, $F$ is algebraic. Thus $F$ extends holomorphically along a generic path. In particular, we can extend $F$ holomorphically to a neighborhood $V=V_1 \times \cdots \times V_k$ of  some boundary point $p_0 \in \partial \mathbb{B}^{n_1} \times \mathbb{B}^{n_2} \times \cdots \times \mathbb{B}^{n_k}$, where each $V_j$ is an open connected subset of $\mathbb{C}^{n_j}.$  By the identity theorem of analytic functions, we conclude that equation (\ref{eqnmetric}) still holds in $V.$

We then restrict (\ref{eqnmetric}) to $(V_1 \cap \partial \mathbb{B}^{n_1}) \times V_{2} \times \cdots  \times V_{k}.$ Since the right hand side is identically zero,  the left hand side of the equation vanishes identically as well. Consequently, we conclude that there is some component  $F_i$ such that its norm is identically $1$ when restricted to $(V_1 \cap \partial \mathbb{B}^{n_1}) \times V_{2} \times \cdots  \times V_{k}.$ By permuting the order of $N_i'$s if necessary, we call it $F_1.$ This implies $F_1$ maps $(V_1 \cap \partial \mathbb{B}^{n_1}) \times V_{2} \times \cdots  \times V_{k}$ into the unit sphere $\partial \mathbb{B}^{N_1}.$ Note that $\partial \mathbb{B}^{N_1}$ contains no non-trivial complex varieties. Therefore $F_1$ is constant in the variables $Z_2, \cdots, Z_k.$ Namely, $F_1$ only depends on $Z_1.$

We now restrict $F$ to $\mathbb{B}^{n_1} \times \{ 0\} \times \cdots \times \{ 0 \}$, by setting $Z_2=\cdots=Z_k={\bf 0}.$ Then the restriction $\widetilde{F}=(\widetilde{F}_1, \cdots, \widetilde{F}_l)$ of $F$ is a holomorphic isometry from $\mathbb{B}^{n_1}$ to $\mathbb{B}^{N_1} \times \cdots \times \mathbb{B}^{N_l}.$ This reduces to the case $k=1.$ By \cite{YZ12}, each component $\widetilde{F}_i$  of $\widetilde{F}$ is a totally geodesic holomorphic isometry to $\mathbb{B}^{N_i}.$ In particular, $F_1=\widetilde{F}_1$ is a totally geodesic holomorphic isometry from $\mathbb{B}^{n_1}$ to $\mathbb{B}^{N_1}.$ Namely, $1-\lVert F_1\rVert^2=1-\lVert Z_1\rVert^2.$ By comparing the vanishing order of both sides in (\ref{eqnmetric}) at $p_0,$ we have $\mu_1 \leq \lambda_1.$ Moreover, we can eliminate $F_1$ in (\ref{eqnmetric}) to get,
$$\prod_{i=2}^l \left( 1-\lVert F_i\rVert^2 \right)^{\mu_i}=(1-\lVert Z_1\rVert^2)^{\lambda_1-\mu_1}\prod_{j=2}^k \left(1-\lVert Z_j\rVert^2 \right)^{\lambda_j}.$$
Therefore we can prove by induction on $l$ that for each $F_i$, there exists some $j$ such that $F_i$ only depends on the coordinates of $\mathbb{B}^{n_j}.$ Moreover, $F_i$ extends to a totally geodesic holomorphic isometric embedding from $(\mathbb{B}^{n_j},g_{\mathbb{B}^{n_j}})$ to $(\mathbb{B}^{N_i},g_{\mathbb{B}^{N_i}}).$ It is easy to see $\sum_{i=1}^l \mu_i= \sum_{j=1}^k \lambda_j$.
\end{proof}

\bigskip

We now consider the general case of holomorphic isometries from the Poincar\'e disk to polydisks with positive conformal constants. 
Let $f=(f_1,\ldots,f_p):(\Delta,g_\Delta)\to (\Delta,\lambda_1 g_\Delta)\times\cdots \times (\Delta,\lambda_p g_\Delta)$ be a holomorphic isometry, where $p\ge 2$ is an integer and $\lambda_j>0$ is a real constant for $j=1,\ldots,p$. Namely, $$\sum_{j=1}^p \lambda_j f_j^*g_\Delta = g_\Delta.$$
Suppose that $f(0)={\bf 0}$.
Then, we have the functional equation
\[ \prod_{j=1}^p (1-|f_j(w)|^2)^{\lambda_j}
= 1-|w|^2. \]
We observe from the functional equation and the Schwarz lemma that $\sum_{j=1}^p \lambda_j \ge 1$.
In addition, if $\sum_{j=1}^p \lambda_j=1$, then $f$ is the diagonal embedding $w\mapsto(w,\ldots,w)$ up to reparametrizations. In particular, if $p=2$, we provide the following generalization of the classification result of Ng in \cite{Ng10}.

\begin{thm}
Let $f=(f_1,f_2):(\Delta,g_\Delta)\to (\Delta,\lambda_1 g_\Delta)\times (\Delta,\lambda_2 g_\Delta)$ be a holomorphic isometry, where $\lambda_1,\lambda_2>0$ are real constants.
Then, $f$ is one of the following up to reparametrizations:
\begin{enumerate}
\item $f(z)= (z,0)$ with $\lambda_1=1$ and $\lambda_2>0$ arbitrary.
\item $f(z)= (0, z)$ with $\lambda_2=1$ and $\lambda_1>0$ arbitrary.
\item $f(z)=(z,z)$ with $\lambda_1+\lambda_2=1$.
\item $f=(\alpha,\beta):\Delta\to \Delta^2$ is the square root embedding with $\lambda_1=\lambda_2=1$.
\end{enumerate}
\end{thm}
\begin{proof}
First we conclude by Theorem \ref{Thm:Algebraicity} that $f$ is algebraic.  We may suppose without loss of generality that $f(0)={\bf 0}$ and thus
we have the functional equation
\begin{equation}\label{albert}
 \left(1-|f_1(z)|^2\right)^{\lambda_1}\left(1-|f_2(z)|^2\right)^{\lambda_2}=1-|z|^2.
\end{equation}
We proceed in the following different cases.

{\bf Case I:} If one of the $f_1$ and $f_2$ is constant, say $f_2\equiv 0$, then we have $\lambda_1=1$, $\lambda_2>0$ is arbitrary. This shows that $f_1(z)=z$ up to reparametrizations and is included in the case (1) and (2).

{\bf Case II:} Neither of the $f_j, j=1,2, $ is constant.  We first establish the following lemma.
\begin{lem}\label{lemmapre}
Fix $j \in \{1, 2\}.$ Assume $b \in \partial \Delta$ is not a branch point of $f$ and $|f_j(z)|\to 1$ as $z\to b$,  then for any sufficiently closed  boundary point $b' \not= b,$ we know that $b'$ is not a branch point of $f$, and the vanishing order of $1-f_j(z)\overline{f_j(b')}$ at $z=b'$ equals to one.
\end{lem}
{\bf Proof of Lemma \ref{lemmapre}:} We note that there are only finitely many branch points of $f$ on $\partial \Delta.$
Since $|f_j(z)| \to 1$ as $z\to b$,  there is a small open neighborhood $U_b$ of $b$ in $\mathbb C$ such that $|f_j(z)|\to 1$ as $z\to b'$ for any $b'\in U_b\cap \partial\Delta$ by the Open Mapping Theorem. We then note $1-f_j(z)\overline{f_j(b')}$ vanishes to the order equals to $1$ at $b$ if and only if the derivative $f_j'(b') \neq 0$.
By shrinking $U_b$ if necessary, we have for any point $b'\in U_b\cap\partial\Delta,$ $f_j'(b')\neq 0$. Indeed,  otherwise there is a sequence $b_j \rightarrow b$ such that  $f_j'( b_j) = 0$, then $f_j'\equiv 0$. Hence $f_j$ is a constant function, which is a contradiction. The lemma is proved.

\bigskip

We further split {\bf  Case II} into several subcases. Let $b\in \partial\Delta$ be  a unbranch point of $f$.  Note when $z \rightarrow b,$ at least one of $|f_j(z)|$ goes to one.

{\bf Case II(a):} There is a point $b\in \partial\Delta$ which is not a branch point of $f$ such that $|f_j(z)|^2 \to 1$ as $z\to b$ for both $j=1,2$. Then perturbing $b$ if necessary,  we may assume that $f_j'(b)\neq 0$ for both $j=1,2$ by Lemma \ref{lemmapre}.
Then, we have $\lambda_1+\lambda_2=1$ by comparing the vanishing order at $b$ in the polarized equation 
\begin{equation}\label{po}
\left(1-f_1(z)\overline{f_1(b)}\right)^{\lambda_1} \left(1-f_2(z) \overline{f_2(b)}\right)^{\lambda_2} = 1-z \bar b.
\end{equation}
As mentioned above, by the Schwarz lemma, $f(z)=(z,z)$ up to reparametrizations. This is the case (3).

It remains the case where for any point $b\in \partial\Delta$ which is not a branch point of $f$, we have  one of $|f_k(z)|^2 \to 1$ as $z\to b$ and the other of $|f_k(z)|^2 \to r < 1$ as $z \rightarrow b$.

{\bf Case II (b):} There are two distinct points $b_1,b_2\in \partial\Delta$ such that $|f_j(z)|^2 \to 1$ as $z\to b_j$ for $j=1,2$. By Lemma \ref{lemmapre}, replacing $b_j$ with another unbranch point $b_j'\in \partial\Delta$ if necessary, we can assume the vanishing order of  $1-f_j(z)\overline{f_j(b_j)}$ equals one at $b_j$ for $j=1,2$. We compare the vanishing order at $b_j$ to get $\lambda_1=1$ and $\lambda_2=1$ using equation (\ref{po}).
Then $f$ is a holomorphic isometry from $(\Delta,ds_\Delta^2)$ to $(\Delta^2,ds_{\Delta^2}^2)$ and we are done by Ng's classification of the $2$-disk \cite{Ng10} (also cf. \cite{XY16a}). This is the case (4).

{\bf Case II (c):} We now consider the remaining case: It is always the same
$|f_{\sigma(1)}(z)|^2 \to 1$ and the other $|f_{\sigma(2)}(z)|^2$ $\to$ $r(b)<1$ as $z\to b$ for 
any unbranch point $b\in \partial\Delta$ of $f$. Here $\sigma\in S_2$ is a fixed permutation.
As pointed out in Mok \cite{Mok11} that $f$ extends continuously from $\overline\Delta$ to $\overline\Delta^2$. This can be proved by using Puiseux series expansion of each component function of $f$ around each branch point of $f$ and the boundedness of the component functions of $f$ on $\Delta$.
This implies that $|f_{\sigma(1)}(z)|^2 \to 1$ for any $b\in \partial\Delta$ because the set of branch points of $f$ on $\partial\Delta$ is finite and $|f_{\sigma(1)}(z)|^2$ is continuous on $\partial\Delta$.
Therefore, $f_{\sigma(1)}:\Delta\to \Delta$ is a proper holomorphic map and thus a finite Blaschke product. Moreover, we have $\lambda_1=1$.
Suppose $f_{\sigma(1)}$ has a pole at $p \in \mathbb{C}$. Let $w_j \rightarrow p$ such that each $w_j$ is not a branch point of $f_{\sigma(2)}$ and $\lim_{w_j\rightarrow p}f_{\sigma(2)}(w_j)$ equals to $\xi \in \mathbb{C}$ or $\lim_{w_j\rightarrow p}|f_{\sigma(2)}(w_j)| = \infty$. Using equation (\ref{po}), $\lim_{w_j\rightarrow p} \left(1-f_{\sigma(2)}(z)\overline{f_{\sigma(2)}(w_j)} \right)=0$ off finitely many $z \in \mathbb{C}$, implying $f_{\sigma(2)}$ is constant. This is contradictory to the hypotheses. Therefore $f_{\sigma(1)}(z) = c z^k$ for $|c|=1, k \in \mathbb{N}$. It follows from equation (\ref{albert}) that 
\begin{equation}\label{um}
\left( \frac{1}{1-|f_{\sigma(2)}(z)|^2} \right)^{\lambda_{\sigma(2)}} = \frac{1-|z|^{2k}}{1-|z|^2}.
\end{equation}
The right hand side of equation (\ref{um}) is a finite sum of the norm squares of holomorphic functions and the left hand side is an infinite sum of the norm squares of linearly independent holomorphic functions unless $f_{\sigma(2)}$ is constant. This is a contradiction by \cite{Um88} to the assumption that neither of $f_j$ is constant. An alternative way to reach the contradiction is using equation (\ref{um}) to see that for any branch of $f_{\sigma(2)}$ along the holomorphic continuation, $|1-|f_{\sigma(2)}(z)|^2 |\leq 1$ on $\mathbb{C}$ off isolated points. Therefore $f_{\sigma(2)}$ must be constant by  Removable Singularity Theorem and Liouville's Theorem.
\end{proof}

{\bf Acknowledgement} We are grateful to the referee for his/her thorough review and helpful suggestions.

\end{document}